\documentclass[a4paper,reqno]{amsart}

\usepackage{amssymb}
\usepackage{amscd}
\usepackage{amsthm}
\usepackage{amsmath}
\usepackage{enumerate}
\usepackage{mathrsfs}
\usepackage{graphics}
\usepackage[all]{xy} 
\usepackage{tikz}
\usepackage{extarrows}
\usepackage{tikz-cd}
\usetikzlibrary{calc}
\usepackage[retainorgcmds]{IEEEtrantools}
\usetikzlibrary{matrix,arrows,decorations.pathmorphing}

\newtheorem{Thm}{Theorem}[section]\newtheorem*{Thm*}{Theorem}

\newtheorem{Lem}[Thm]{Lemma} 
\newtheorem{Cor}[Thm]{Corollary}

\newtheorem{Prop}[Thm]{Proposition}
\newtheorem{Prop-Def}[Thm]{Proposition-Definition}

\theoremstyle{definition}
\newtheorem{Ex}[Thm]{Example}
\newtheorem{Def}[Thm]{Definition}
\newtheorem{Assu}[Thm]{Assumption}
\newtheorem{Rem}[Thm]{Remark}

\newcommand{\ra}{\rightarrow}

\newcommand{\D}{\mathcal{D}}
\newcommand{\C}{\mathcal{C}}

\newcommand{\SSS}{\mathbb{S}}

\newcommand{\cal}{\mathcal}
\newcommand{\T}{\mathcal T}
\newcommand{\X}{\mathcal X}
\newcommand{\Y}{\mathcal Y}
\newcommand{\U}{\mathcal U}
\newcommand{\V}{\mathcal V}

\newcommand{\bb}{\mathrm{b}}

\newcommand{\Filt}{\mathsf{Filt}}

\newcommand{\h}{{\mathrm H}}
\renewcommand{\H}{{\mathcal H}}
\newcommand{\xra}{\xrightarrow}
\newcommand{\Z}{{\mathbb Z}}

\newcommand{\rev}{\rm rev}

\newcommand{\op}{\oplus}
\newcommand{\bop}{\bigoplus}
\newcommand{\ot}{\otimes}

\newcommand{\hs}{\hspace{-3pt}}

\newcommand{\Hom}{\operatorname{Hom}\nolimits}

\newcommand{\End}{\operatorname{End}\nolimits}
\newcommand{\RHom}{\mathbf{R}\strut\kern-.2em\operatorname{Hom}\nolimits}
\newcommand{\RshHom}{\mathbf{R}\strut\kern-.2em\mathscr{H}\strut\kern-.3em\operatorname{om}\nolimits}
\newcommand{\shHom}{\mathscr{H}\strut\kern-.3em\operatorname{om}\nolimits}
\newcommand{\shEnd}{\mathscr{E}\strut\kern-.3em\operatorname{nd}\nolimits}

\DeclareMathOperator{\moduleCategory}{{\mathsf{mod}}} \renewcommand{\mod}{\moduleCategory}
\DeclareMathOperator{\proj}{\mathsf {proj}}
\DeclareMathOperator{\thick}{\mathsf{thick}}
\DeclareMathOperator{\per}{\mathsf{per}}
\DeclareMathOperator{\add}{\mathsf {add}}

\DeclareMathOperator{\Inj}{\mathsf{inj}}

\DeclareMathOperator{\ind}{\mathsf {ind}}
\DeclareMathOperator{\silt}{\mathsf {silt}}

\DeclareMathOperator{\ctilt}{{\mbox{$(d+1)$}}\mathsf{-ctilt}}

\DeclareMathOperator{\silheart}{\mathsf {silt-heart}}
\DeclareMathOperator{\siltexc}{\mathsf {silt-exc}}
\DeclareMathOperator{\SMCexc}{\mathsf {SMC-exc}}
\DeclareMathOperator{\fullexc}{\mathsf {full-exc}}

\DeclareMathOperator{\maxsim}{\mathsf {max-sim}}
\DeclareMathOperator{\pmaxsim}{\mathsf {max-sim^{+}}}

\DeclareMathOperator{\CYconf}{\mathsf {CY-conf}}
\DeclareMathOperator{\SMC}{\mathsf {SMC}}
\DeclareMathOperator{\SMS}{{\mbox{$d$}}\mathsf {-SMS}}

\numberwithin{equation}{section}

\title[Positive Fuss-Catalan number and simple-minded systems]{Positive Fuss-Catalan numbers and Simple-minded systems \\ in negative Calabi-Yau categories}

\author{Osamu Iyama}
\address{O. Iyama: Graduate School of Mathematical Sciences, University of Tokyo, 3-8-1 Komaba Meguro-ku Tokyo 153-8914, Japan}
\email{iyama@ms.u-tokyo.ac.jp}

\address{O. Iyama: Graduate School of Mathematics, Nagoya University, Chikusa-ku, Nagoya, 464-8602 Japan}

\author{Haibo Jin}
\address{H. Jin: Mathematical Institut, University of Cologne, Weyertal 86-90, 50931 K\"oln, Germany}
\email{hjin@math.uni-koeln.de}

\address{H. Jin: Graduate School of Mathematics, Nagoya University, Chikusa-ku, Nagoya, 464-8602 Japan}

\begin{document}

\begin{abstract}
We  establish a bijection between $d$-simple-minded systems ($d$-SMSs)  of $(-d)$-Calabi-Yau cluster category $\cal C_{-d}(H)$
and silting objects of $\D^{\bb}(H)$ contained in $\D^{\le 0}\cap \D^{\ge 1-d}$ for hereditary algebra $H$ of Dynkin type and $d\ge 1$. 
We show that the number of $d$-SMSs in $\cal C_{-d}(H)$ is the positive Fuss-Catalan number 
$C_{d}^{+}(W)$ of the corresponding Weyl group $W$, by applying this bijection and
Buan-Reiten-Thomas' and Zhu's results on
 Fomin-Reading's  generalized cluster complexes. Our results are based on a refined version of silting-$t$-structure correspondence. 
\end{abstract}

\thanks{The first author is supported by JSPS Grant-in-Aid for Scientific Research (B) 16H03923, (C) 18K3209 and (S) 15H05738. The second author was  supported by China Scholarship Council No.201606140033.}

\maketitle

\section{Introduction}

Fomin and Zelevinsky  \cite{FZ} showed that cluster algebras of finite type correspond bijectively with finite root systems $\Phi$.
As a generalization of   their combinatorial structure,   Fomin and Reading \cite{FR}  introduced generalized cluster complex $\Delta^{d}(\Phi)$  for each positive integer $d$. It is a simplicial complex whose ground set is the disjoint union of $d$ copies of the set $\Phi^{+}$ of positive roots and the set of negative simple roots,
and studied actively in combinatorics \cite{A,STW}.
It is known that $\Delta^{d}(\Phi)$  is categorified  by 
$(d+1)$-Calabi-Yau ($(d+1)$-CY) cluster categories $\cal C_{d+1}(H)$
\footnote{$(d+1)$-CY cluster categories are usually called $d$-cluster categories in hereditary setting, and $(d+1)$-cluster categories in non-hereditary setting.}
    for the corresponding hereditary algebra $H$ of Dynkin type \cite{Keller05,T}.
The category $\C_{d+1}(H)$ has
  special objects called $(d+1)$-cluster tilting  objects (see Definition \ref{Def:csilt}), which correspond bijectively with maximal simplices  in $\Delta^{d}(\Phi)$ \cite{Z} and with silting objects (see Definition \ref{Def:silting}) contained in some subcategory of $\D^{\rm b}(H)$ \cite{BRT2}. 
Cluster tilting objects also play a key role in Cohen-Macaulay representations \cite{I}.
 
Recently
there is  increasing interest in negative CY triangulated categories (Definition \ref{Def:CYtri}) (see \cite{CS0,CS,CS3,CSP,HJY, J1,J2,Jo,KYZ}), including  $(-d)$-CY cluster categories $\cal C_{-d}(H)$ (see Section \ref{section:1.1}). These categories often contain special objects called $d$-simple-minded systems (or $d$-SMS) \cite{CS} (see Definition \ref{Def:CY}). 
SMS  plays a key role in the study of Cohen-Macaulay dg modules \cite{J1}. 

{\small
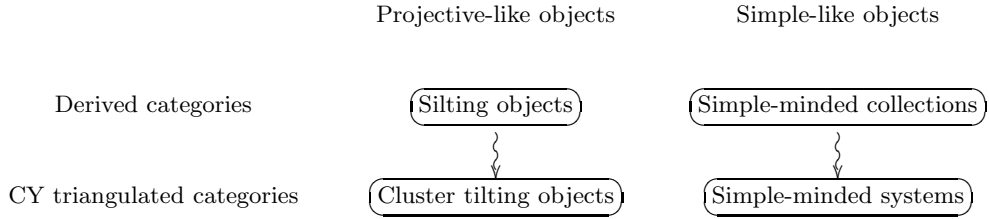
\[\begin{xy}
(30,32)*{\txt{Projective-like objects}},
(-15,20)*\txt{Derived categories},
(30,20)*+[F:<5pt>]{\txt{Silting objects}}="1",
(-15,8)*{\txt{CY  triangulated categories}},
(30,8)*+[F:<5pt>]{\txt{Cluster tilting objects}}="2",
(75,32)*{\txt{Simple-like objects}},
(75,20)*+[F:<5pt>]{\txt{Simple-minded collections}}="3",
(75,8)*+[F:<5pt>]{\txt{Simple-minded  systems}}="4",
\ar@{~>}"1";"2" 
\ar@{~>}"3";"4" 
\end{xy}\]}
In some important cases, cluster tilting objects  and $d$-SMSs are shadows of more fundamental objects, namely, silting  objects (Definition \ref{Def:silting}) and simple-minded collections (SMCs) (see Definition \ref{Def:SMC}). 

The aim of this paper is to show that there is a bijection between $d$-SMSs and maximal simplices in $\Delta^{d}(\Phi)$ without negative simple roots. 
 In particular, the number of $d$-SMSs {(Definition \ref{Def:CY})} in $\cal C_{-d}(H)$ is precisely the   positive Fuss-Catalan number. 
 Our method is based on a refined version of silting-$t$-structure correspondence, which is a bijection between silting objects in perfect derived categories  of a finite-dimensional algebra $A$ and SMCs in bounded derived categories of $A$.  

\subsection{Counting $d$-simple-minded systems}\label{section:1.1}
Let $\Phi$ be a simply-laced finite root system, and $W$ the corresponding Weyl group.
Let $\Phi_{+}$ be the set of positive roots and let $\alpha_{i}$ ($i\in I$)  be the simple root.
The \emph{Fuss-Catalan number} is 
given by the formula
\[ C_{d}(W):=\prod_{i=1}^{n}\frac{dh+e_{i}+1}{e_{i}+1} \]
where $n$ is the rank of $W$, $h$ is its Coxeter number, and $e_{1},\ldots, e_{n}$ are its exponents (see Figure \ref{figure}). 
Recall that the generalized cluster complex $\Delta^{d}(\Phi)$ is a simplicial complex, whose ground set is 
\[ \Phi_{\ge -1}^{d} = (\Phi_{+}\times \{1, \dots, d\})
\cup \{(-\alpha_{i}, 1) \mid i\in I\},\]
and a simplex is a subset of $\Phi_{\ge -1}^{d}$ satisfying  a compatiblility condition (see \cite[Definition 3.1]{FR}).

It is well-known that $C_{d}(W)$  equals the number of maximal simplices in $\Delta^{d}(\Phi)$
and also equals the number of $d$-noncrossing partitions for $W$ (see  \cite{A, FR}).
There is  a variant of $C_{d}(W)$, called the \emph{positive Fuss-Catalan number}, denoted by $C_{d}^{+}(W)$ and given by the  formula 
\[ C_{d}^{+}(W):=\prod_{i=1}^{n}\frac{dh+e_{i}-1}{e_{i}+1}, \]
see Figure \ref{figure} for the explicit value.

 Let $k$ be a field and   let $H$ be a hereditary $k$-algebra of Dynkin type.
 For an integer $d$,
the \emph{$(-d)$-cluster category  $\cal C_{-d}(H)$} is defined as the orbit category $\cal C_{-d}(H):=\cal D^{\rm b}(H)/\nu[d]$, where $\nu {:= ?\ot_{H}^{\bf L}D H}$ is the Nakayama functor of $\D^{\rm b}(H)$. This is a triangulated category by \cite{Keller05} and has AR quiver $\Z Q/\nu[d]$  for the valued quiver $Q$ of $H$. We denote by $\SMS  \cal C_{-d}(H)$ the set of  isomorphic classes of $d$-SMSs {(see Definition \ref{Def:CY})} in  $\cal C_{-d}(H)$,
  and by  $\maxsim\Delta^{d} (\Phi)$ (resp. $\pmaxsim\Delta^{d}(\Phi)$) the set of maximal simplices (resp. maximal simplices without negative simple roots)  in $\Delta^{d}(\Phi)$.
We will prove the following result.

\begin{Thm}[Theorem \ref{Thm:FC}] \label{Thm:FCA}
Let $H$ be a hereditary $k$-algebra of Dynkin type and let $d\ge 1$.
\begin{enumerate}[\rm (1)]
 \item There is a bijection
 \[ \SMS \cal C_{-d}(H)\xleftrightarrow{1:1} \pmaxsim\Delta^{d}(\Phi).\]
 \item
We have 
$\#\SMS\cal C_{-d}(H)=C_{d}^{+}(W)$,
where $W$ is the Weyl group of $H$. 
\end{enumerate}
\end{Thm}

The result (2) is known for the case $d=1$ by \cite{CS0} and for type $A_{n}$  by \cite{J1}.
Figure \ref{figure} gives us concrete formulas for each Dynkin type.

{\renewcommand{\arraystretch}{1.5} 
\begin{figure}\begin{tabular}{|c||c|c|c|}
\hline 
 $Q$ & $h$ & $e_{1}, \dots, e_{n}$ & 
 $C_{d}^{+}(W)=\#\SMS  \cal C_{-d}(H)$\\
\hline \hline
$A_{n}$ & $n+1$ & $1,2, \dots, n$ &  $\frac{1}{n+1}\binom{(d+1)n+d-1}{n}$ \\ 
\hline
$B_{n}, C_{n}$ & $2n$ &  $1,3, \dots, 2n-1$ &   $\binom{(d+1)n-1}{n}$ \\ 
\hline
$D_{n}$ & $2(n-1)$ & $1,3,\dots, 2n-3, n-1$& 
 $\frac{(2d+1)n-2d-2}{n}\binom{(n-1)(d+1)-1}{n-1}$ \\
 \hline
 $E_{6}$ & 12 & $1,4,5,7,8,11$ & $\frac{d(2d+1)(3d+1)(4d+1)(6d+5)(12d+7)}{30}$ \\
 \hline
 $E_{7}$ & $18$ & $1,5,7,9,11,13,17$ & 
$ \frac{d(3d+1)(3d+2)(9d+2)(9d+4)(9d+5)(9d+8)}{280}$\\
 \hline 
 $E_{8}$ & $30$ & $1,7,11,13,17,19,23,29$ & 
 $\frac{d(3d+1)(5d+1)(5d+2)(5d+3)(15d+8)(15d+11)(15d+14)}{1344}$
 \\ \hline
  $F_{4}$ &  $12$ & $1,5,7,11$ & 
 $\frac{d(2d+1)(3d+1)(6d+5)}{2}$
 \\ \hline
 $G_{2}$ & $6$ & $1,5$ & 
 $3d^{2}+2d$
 \\ \hline
\end{tabular}\caption{Positive Fuss-Catalan numbers} \label{figure}
\end{figure}}

To prove Theorem \ref{Thm:FC}, we need to introduce some categorical notions.
We define the following subcategories of $\D^{\bb}(H)$ for any $n, m\in \Z$.
\[ \D^{\le n}:=\{ X\in\D^{\rm b}(H)\mid \h^{>n}(X)=0\}, \  \D^{\ge m}:=\{ X\in\D^{\rm b}(H)\mid \h^{<m}(X)=0\}. \]
We have  standard $t$-structures {(see Definition \ref{Def:t-cot})} 
$\D^{\bb}(H)=\D^{\le n}\perp \D^{\ge n+1}$ for any $n\in \Z$. 

Let $m\le n$. Since we have 
\begin{equation}\label{nu}
\D^{\le m-1}\subset\nu \D^{\le m}\subset\D^{\le m}
 \ \mbox{and}  \   \D^{\ge m+1}\subset \nu^{-1}\D^{\ge m}\subset \D^{\ge m},\end{equation}
we  define three subcategories by
\[ \D^{[m,n]}:=\D^{\le n}\cap \D^{\ge m} \subset  
\D^{[m,n]}_{-}:=\D^{\le n}\cap \nu\D^{\ge m+1} \subset \D_{+}^{[m,n]}:=\D^{\le n}\cap \nu^{-1}\D^{\ge m-1}.  \]
Then $\D_{+}^{[1-d, 0]}$  and $\D_{-}^{[-d, 0]}$ are the fundamental domains of $\C_{d+1}(H)$  and $\C_{-d}(H)$ respectively. More precisely, the canonical functors $\D_{+}^{[1-d,0]} \rightarrow \C_{d+1}(H)$ and $\D_{-}^{[-d,0]} \rightarrow \C_{-d}(H)$ induce bijections
\begin{equation} \label{equation:ind+} \ind \D_{+}^{[1-d,0]}=\bigcup_{i=0}^{d-1}\ind (\mod H)[i] \bigcup \ind (\proj H)[d]\xlongrightarrow{\simeq} \ind \C_{d+1}(H),    \end{equation}
\begin{equation} \label{equation:ind-}
  \ind\D_{-}^{[-d,0]}=  \bigcup_{i=0}^{d-1}\ind (\mod H)[i] \bigcup \ind(\mod H\backslash \Inj H)[d] \xlongrightarrow{\simeq} \ind\C_{-d}(H) \end{equation}
respectively (see \cite{BMRRT}),  where
 $\ind$ denotes the set of isomorphism classes of indecomposable objects.

We denote by $\silt \D^{\rm b}(H)$ (resp. $\SMC \D^{\rm b}(H)$) the set of isomorphic classes of basic silting objects (resp. SMCs (see Definition \ref{Def:SMC})) in   $\D^{\rm b}(H)$. 
The following is a main result of this paper.
\begin{Thm}[Theorem \ref{Thm:bij}]\label{Thm:bijA}
Let $H$ be a hereditary $k$-algebra of  Dynkin type and $d\ge 1$. Then there are  bijections
\begin{eqnarray}
(\silt\D^{\rm b}(H))\cap \cal D^{[1-d,0]} 
&\xleftrightarrow{1:1} &
(\SMC\D^{\rm b}(H))\cap \D^{[-d, 0]}_{-}, \label{silttoSMC}
\\
&\xleftrightarrow{1:1}& 
\SMS\cal C_{-d}(H), \label{SMCtoSMS}
\end{eqnarray}
where the map \eqref{silttoSMC} is induced by the natural functor $ \D^{\rm b}(H) \xra{\pi}  \cal D^{\rm b}(H)/\nu[d].$

\end{Thm}


The bijection \eqref{silttoSMC} holds in a more general setting, that is, it is true for any finite-dimensional Iwanaga-Gorenstein algebras (which will be proved in Corollary \ref{Cor:Gor} and used in the proof of Theorem \ref{Thm:bijA}). 
Notice that, in the recent paper \cite{CSPP}, the  bijection \eqref{SMCtoSMS} was given for the path algebra $kQ$ of any acyclic quiver $Q$.

Our theorems and  the results in \cite{BRT2, Z} mentioned above  are summarized as follows,
where we denote by $\ctilt \C_{d+1}(H)$ 
 the set of  $(d+1)$-cluster tilting objects in  $\C_{d+1}(H)$.

\[{ \xymatrixcolsep{2.5pc}\xymatrix{ 
  \SMS\cal C_{-d}(H) \ar@{<->}[rr]_{\sim}^{\text{Thm.} \ref{Thm:bij}} & &(\SMC\D^{\rm b}(H))\cap \D^{[-d,0]}_{-} \ar@{<->}[d]^{\wr}_{\text{Thm.} \ref{Thm:bij}}
\\
\pmaxsim \Delta^{d} (\Phi) 
 \ar@{<->}[rr]_{\sim}  \ar@{^{(}->}[d] \ar@{<->}[u]_{\wr}^{\text{Thm.} \ref{Thm:FC}}& & (\silt \D^{\rm b}(H))\cap \D^{[1-d,0]}\ar@{^{(}->}[d]  \\
   \maxsim\Delta^{d} (\Phi) 
\ar@{<->}[r]_(0.5){\sim}^(0.5){\mbox{\footnotesize \cite{Z}}} & \ctilt \C_{d+1}(H)  \ar@{<->}[r]^(0.42){\mbox{\footnotesize \cite{BRT2}}}_(.42){\sim} & (\silt \D^{\rm b}(H)) \cap \D^{[1-d,0]}_{+}
.}}\]

We give an example of Theorem \ref{Thm:bij}. Recall the cluster category $\C_{-d}(H)$ has AR quiver $\Z Q/\nu[d]$ for the valued quiver $Q$ of $H$.
\begin{Ex}
\begin{enumerate}[\rm(1)]
\item
Let $H=kA_{3}$  and $d=1$, then the bijection between silting objects of $\D^{\rm b}(kA_{3})$ contained in $\mod kA_{3}=\D^{[0,0]}$ and $1$-SMS of $\cal C_{-1}(kA_{3})$ is as follows.
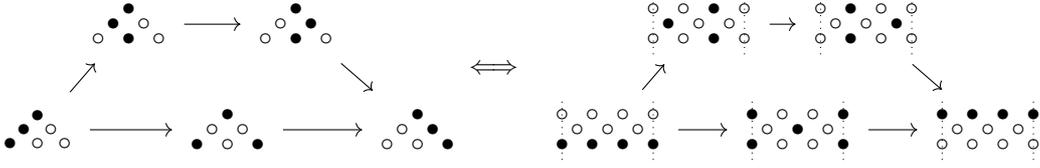
\begin{center}
\begin{tikzpicture}
\draw node (1) at (0.5,0){\begin{tikzpicture}[scale=0.18]
\draw node  at (0,0){$\bullet$} node  at (1,1){$\bullet$} node  at (2,2){$\bullet$} node  at (2,0){$\circ$} node  at (4,0){$\circ$} node  at (3,1){$\circ$};\end{tikzpicture}}

node (2) at (1.7,1.4){\begin{tikzpicture}[scale=0.2]  \draw node  at (0,0){$\circ$} node  at (1,1){$\bullet$} node  at (2,2){$\bullet$} node  at (2,0){$\bullet$} node  at (4,0){$\circ$} node  at (3,1){$\circ$};\end{tikzpicture}}

node (3) at (3.9,1.4){\begin{tikzpicture}[scale=0.2]  \draw node  at (0,0){$\circ$} node  at (1,1){$\circ$} node  at (2,2){$\bullet$} node  at (2,0){$\bullet$} node  at (4,0){$\circ$} node  at (3,1){$\bullet$};\end{tikzpicture}}  

node (4) at (5.5,0){\begin{tikzpicture}[scale=0.2]  \draw node  at (0,0){$\circ$} node  at (1,1){$\circ$} node  at (2,2){$\bullet$} node  at (2,0){$\circ$} node  at (4,0){$\bullet$} node  at (3,1){$\bullet$};\end{tikzpicture}}

node (5) at (3,0){\begin{tikzpicture}[scale=0.2]  \draw node  at (0,0){$\bullet$} node  at (1,1){$\circ$} node  at (2,2){$\bullet$} node  at (2,0){$\circ$} node  at (4,0){$\bullet$} node  at (3,1){$\circ$};\end{tikzpicture}}
[->](1)edge(2) (2)edge(3) (3)edge(4) (5)edge(4)  (1)edge(5);

\draw[xshift=7.5cm]
node (1) at (0.5,0){\begin{tikzpicture}[scale=0.2]
\draw node at (0,0){$\bullet$}  node at (2,0){$\bullet$}  node at (4,0){$\bullet$}  node at (6,0){$\bullet$}  node at (1,1){$\circ$}  node at (3,1){$\circ$}  node at (5,1){$\circ$}  node at (0,2){$\circ$}  node at (2,2){$\circ$}  node at (4,2){$\circ$}  node at (6,2){$\circ$}
 [dotted](0,-1)--(0,3) (6,-1)--(6,3);\end{tikzpicture}}

node (2) at (1.7,1.4){\begin{tikzpicture}[scale=0.2]
\draw node at (0,0){$\circ$}  node at (2,0){$\circ$}  node at (4,0){$\bullet$}  node at (6,0){$\circ$}  node at (1,1){$\bullet$}  node at (3,1){$\circ$}  node at (5,1){$\circ$}  node at (0,2){$\circ$}  node at (2,2){$\circ$}  node at (4,2){$\bullet$}  node at (6,2){$\circ$}
[dotted] (0,-1)--(0,3) (6,-1)--(6,3);\end{tikzpicture}}

node (3) at (3.9,1.4){\begin{tikzpicture}[scale=0.2]
\draw node at (0,0){$\circ$}  node at (2,0){$\bullet$}  node at (4,0){$\circ$}  node at (6,0){$\circ$}  node at (1,1){$\circ$}  node at (3,1){$\circ$}  node at (5,1){$\bullet$}  node at (0,2){$\circ$}  node at (2,2){$\bullet$}  node at (4,2){$\circ$}  node at (6,2){$\circ$}
[dotted] (0,-1)--(0,3) (6,-1)--(6,3);\end{tikzpicture}}

node (4) at (5.5,0){\begin{tikzpicture}[scale=0.2]
\draw node at (0,0){$\circ$}  node at (2,0){$\circ$}  node at (4,0){$\circ$}  node at (6,0){$\circ$}  node at (1,1){$\circ$}  node at (3,1){$\circ$}  node at (5,1){$\circ$}  node at (0,2){$\bullet$}  node at (2,2){$\bullet$}  node at (4,2){$\bullet$}  node at (6,2){$\bullet$}
[dotted] (0,-1)--(0,3) (6,-1)--(6,3);\end{tikzpicture}}

node (5) at (3,0){\begin{tikzpicture}[scale=0.2]
\draw node at (0,0){$\bullet$}  node at (2,0){$\circ$}  node at (4,0){$\circ$}  node at (6,0){$\bullet$}  node at (1,1){$\circ$}  node at (3,1){$\bullet$}  node at (5,1){$\circ$}  node at (0,2){$\bullet$}  node at (2,2){$\circ$}  node at (4,2){$\circ$}  node at (6,2){$\bullet$}
[dotted] (0,-1)--(0,3) (6,-1)--(6,3);\end{tikzpicture}}

[->](1)edge(2) (2)edge(3) (3)edge(4) (5)edge(4)  (1)edge(5);

\draw node at (6.5,0.8){$\Longleftrightarrow$};
\end{tikzpicture}
\end{center}
The five diagrams in left (resp. right) part  are the AR quivers of $\mod kA_{3}$ (resp. $\C_{-1}(kA_{3})$), where black vertices show all elements    of $(\silt \cal D^{\rm b}(kA_{3}))\cap \mod kA_{3}$ (resp. $\mbox{$1$-}\mathsf {SMS}\, \cal C_{-1}(kA_{3})$), and the arrows are  given by mutation (see \cite{BY}). 

\item Let $H=kA_{2}$ and $d=2$. Then the bijection is as follows.
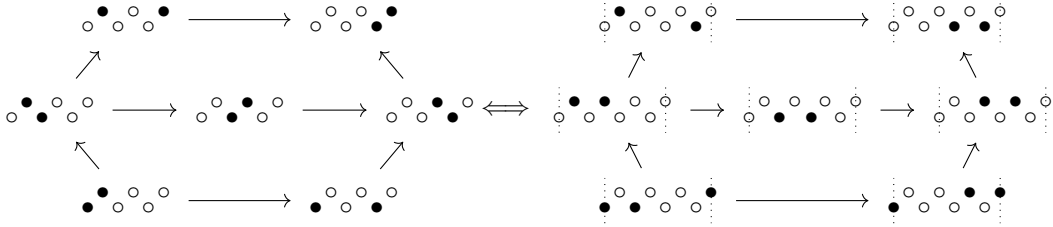
\begin{center}
\begin{tikzpicture}
\draw node (1) at (0,0){ \begin{tikzpicture}[scale=0.2] \draw node at (0,0){$\bullet$} node at (2,0){$\circ$} node at (4,0){$\circ$} node at (1,1){$\bullet$} node at (3,1){$\circ$} node at (5,1){$\circ$};\end{tikzpicture}} 
 
 node (3) at (-1,1.2){ \begin{tikzpicture}[scale=0.2] \draw node at (0,0){$\circ$} node at (2,0){$\bullet$} node at (4,0){$\circ$} node at (1,1){$\bullet$} node at (3,1){$\circ$} node at (5,1){$\circ$};\end{tikzpicture}} 
 
 node (4) at (0,2.4){ \begin{tikzpicture}[scale=0.2] \draw node at (0,0){$\circ$} node at (2,0){$\circ$} node at (4,0){$\circ$} node at (1,1){$\bullet$} node at (3,1){$\circ$} node at (5,1){$\bullet$};\end{tikzpicture}} 

node (2) at (3,0){ \begin{tikzpicture}[scale=0.2] \draw node at (0,0){$\bullet$} node at (2,0){$\circ$} node at (4,0){$\bullet$} node at (1,1){$\circ$} node at (3,1){$\circ$} node at (5,1){$\circ$};\end{tikzpicture}} 

node (5) at (4,1.2){ \begin{tikzpicture}[scale=0.2] \draw node at (0,0){$\circ$} node at (2,0){$\circ$} node at (4,0){$\bullet$} node at (1,1){$\circ$} node at (3,1){$\bullet$} node at (5,1){$\circ$};\end{tikzpicture}} 

node (6) at (3,2.4){ \begin{tikzpicture}[scale=0.2] \draw node at (0,0){$\circ$} node at (2,0){$\circ$} node at (4,0){$\bullet$} node at (1,1){$\circ$} node at (3,1){$\circ$} node at (5,1){$\bullet$};\end{tikzpicture}} 

node (7) at (1.5,1.2){ \begin{tikzpicture}[scale=0.2] \draw node at (0,0){$\circ$} node at (2,0){$\bullet$} node at (4,0){$\circ$} node at (1,1){$\circ$} node at (3,1){$\bullet$} node at (5,1){$\circ$};\end{tikzpicture}} 
[->](1)edge(3) (1)edge(2) (2)edge(5) (5)edge(6) (3)edge(4) (4)edge (6) (3)edge(7) (7)edge (5);

\draw[xshift=7cm]
node (1) at (0,0){\begin{tikzpicture}[scale=0.2]
 \draw node at (0,0){$\bullet$} node at (2,0){$\bullet$} node at (4,0){$\circ$} node at (6,0){$\circ$} node at (1,1){$\circ$} node at (3,1){$\circ$} node at (5,1){$\circ$} node at (7,1){$\bullet$}[dotted] (0,-1)--(0,2) (7,2)--(7,-1);\end{tikzpicture}}
 
 node (3) at (-0.6,1.2){\begin{tikzpicture}[scale=0.2]
 \draw node at (0,0){$\circ$} node at (2,0){$\circ$} node at (4,0){$\circ$} node at (6,0){$\circ$} node at (1,1){$\bullet$} node at (3,1){$\bullet$} node at (5,1){$\circ$} node at (7,1){$\circ$}[dotted] (0,-1)--(0,2) (7,2)--(7,-1);\end{tikzpicture}}
 
 node (4) at (0,2.4){ \begin{tikzpicture}[scale=0.2]
 \draw node at (0,0){$\circ$} node at (2,0){$\circ$} node at (4,0){$\circ$} node at (6,0){$\bullet$} node at (1,1){$\bullet$} node at (3,1){$\circ$} node at (5,1){$\circ$} node at (7,1){$\circ$}[dotted] (0,-1)--(0,2) (7,2)--(7,-1);\end{tikzpicture}}

node (2) at (3.8,0){ \begin{tikzpicture}[scale=0.2]
 \draw node at (0,0){$\bullet$} node at (2,0){$\circ$} node at (4,0){$\circ$} node at (6,0){$\circ$} node at (1,1){$\circ$} node at (3,1){$\circ$} node at (5,1){$\bullet$} node at (7,1){$\bullet$}[dotted] (0,-1)--(0,2) (7,2)--(7,-1);\end{tikzpicture}}
 
 node (5) at (4.4,1.2){ \begin{tikzpicture}[scale=0.2]
 \draw node at (0,0){$\circ$} node at (2,0){$\circ$} node at (4,0){$\circ$} node at (6,0){$\circ$} node at (1,1){$\circ$} node at (3,1){$\bullet$} node at (5,1){$\bullet$} node at (7,1){$\circ$}[dotted] (0,-1)--(0,2) (7,2)--(7,-1);\end{tikzpicture}}
 
 node (6) at (3.8,2.4){ \begin{tikzpicture}[scale=0.2]
 \draw node at (0,0){$\circ$} node at (2,0){$\circ$} node at (4,0){$\bullet$} node at (6,0){$\bullet$} node at (1,1){$\circ$} node at (3,1){$\circ$} node at (5,1){$\circ$} node at (7,1){$\circ$}[dotted] (0,-1)--(0,2) (7,2)--(7,-1);\end{tikzpicture}}
 
 node (7) at (1.9, 1.2){ \begin{tikzpicture}[scale=0.2]
 \draw node at (0,0){$\circ$} node at (2,0){$\bullet$} node at (4.1,0){$\bullet$} node at (6,0){$\circ$} node at (1,1){$\circ$} node at (3,1){$\circ$} node at (5,1){$\circ$} node at (7,1){$\circ$}[dotted] (0,-1)--(0,2) (7,2)--(7,-1);\end{tikzpicture}}
[->](1)edge(3) (1)edge(2) (2)edge(5) (5)edge(6) (3)edge(4) (4)edge (6) (3)edge(7) (7)edge (5);
\draw node at (5,1.2){$\Longleftrightarrow$};
\end{tikzpicture} \end{center}

\end{enumerate}
\end{Ex}


Let $H$ be a hereditary $k$-algebra. 
Recall that \cite{BRT} introduced the notion of 
\emph{$\Hom_{\le 0}$-configurations} of $\D^{\rm b}(H)$ {(see Definition \ref{Def:BRT})}
and  they gave a bijection between silting objects and $\rm{Hom}_{\le 0}$-configurations in \cite[Theorem 2.4]{BRT}.  
This notion is similar to SMC, but  quite different in non-hereditary case.
At the end of this paper, we prove the following
\begin{Thm}[Theorem \ref{Thm:appendixA}]\label{Thm:appendix1A}
Let $H$ be a hereditary $k$-algebra. Then $\rm{Hom}_{\le 0}$-configurations of $\D^{\rm b}(H)$ are precisely SMCs  of $\D^{\rm b}({H})$.
\end{Thm}

One may also deduce the bijection \eqref{silttoSMC} from   Theorem \ref{Thm:appendix1A} and \cite[Theorem 2.4]{BRT}.


\subsection{Silting-$t$-structure correspondence} \label{section:1.2}

There is a useful structure in a triangulated category called $t$-structure (see Definition \ref{Def:t-cot}). In the derived category of a finite dimensional algebra, we have the following important bijection between silting objects and certain $t$-structures.
\begin{Thm}\cite{KY}\label{Thm:KYbij}
Let $A$ be a finite-dimensional $k$-algebra. Then there are bijections, 
\[\silt \per A \xrightarrow{1:1} \{\text{bounded $t$-structures of $\D^{\rm b}(A)$ with length hearts }\} \xrightarrow{1:1} \SMC \D^{\rm b}(A).\]
where the first map sends $P\in \silt \per A$ to the $t$-structure $\D^{\bb}(A)=P[<\hs 0]^{\perp}\perp P[\ge \hs 0]^{\perp}$, and the second map sends a $t$-structure  to the simple objects in the heart.
\end{Thm}
In this subsection, we give  two refined  versions of Theorem \ref{Thm:KYbij}  for triangulated categories,  both of which imply the bijection  \eqref{silttoSMC} above.
Our common assumption is the following, which is satisfied  for $\T=\D^{\rm b}(A)$ and the perfect derived category  $\U=\per A$    for a finite-dimensional $k$-algebra $A$.

\begin{Assu} \label{Assu}
Let $\T$ be a triangulated category and 
  $\U$   a thick subcategory of $\T$ (that is, $\U$ is a triangulated subcategory of $\T$ closed under direct summands). 
 Assume that for any $P\in \silt \U$,   we have a bounded $t$-structure 
 \begin{equation}\T=\T_{ P}^{\le 0}\perp\T_{ P}^{>0}, \text { where $\T_{P}^{\le 0}:=P[<\hs 0]^{\perp}$ and $\T_{P}^{>0}:=P[\ge\hs 0]^{\perp}$}. \label{*}\end{equation}
 See Section \ref{Section:pre} for  the definitions of  $\perp$ and $(\, )^{\perp}$.
\end{Assu}

We call   \eqref{*} the \emph{silting $t$-structure} associated with $P$, and  call its heart $\cal H_{P}:=\T_{P}^{\le 0}\cap \T_{P}^{\ge 0}$  the \emph{silting heart}. 
Then $P$ can be recovered from the subcategory $\cal H_{P}$ (see  Lemma \ref{Lem:poset}).
Denote by $\silheart \T$  the set of silting hearts  of $\T$. Notice that $\silt \U$ and $\silheart \T$ have canonical partial orders (see Section \ref{Section:pre}).

 \begin{Thm}[Theorem \ref{Thm:silt}]\label{Thm:siltA}
 Under Assumption \ref{Assu},
let  $\T=\X\perp\X^{\perp}=\Y\perp \Y^{\perp}$
be  two silting $t$-structures.  Then there is a poset isomorphism
\[ 
\{P\in \silt\U\mid P\in \X\cap {}^{\perp}\Y \} \xleftrightarrow{\cong}  \{ \H \in\silheart \T\mid \H \subset \X\cap \Y^{\perp}\}.\]
 \end{Thm}

For a finite-dimensional $k$-algebra $A$ and $d\ge 1$, we call $P\in \silt \per A$ a  \emph{$d$-term silting},  if  $\Hom_{\cal D^{\rm b}(A)}(A[<\hs0], P)=0=\Hom_{\cal D^{\rm b}(A)}(P, A[\ge\hs d])$.
  For $\T=\D^{\bb}(A)$, by applying Theorem \ref{Thm:silt}
 to  $\X=\D^{\le 0}$ and $\Y=\D^{\le -d}$, we get the corollary below.
It  is well-known for  the case $d=2$ \cite{BY} and plays an important role in cluster theory.


\begin{Cor}[Corollary \ref{Cor:d-term}] \label{Cor:d-termA}
  There is a poset isomorphism.
\[ \{\text{$d$-term silting objects in $\per A$}\}\xleftrightarrow{\cong} (\SMC\D^{\rm b}(A))\cap \D^{[1-d, 0]}.
\]
\end{Cor}

When $\T$ has relative Serre functor in the following sense, we can improve Theorem \ref{Thm:silt} by dropping the assumption that two $t$-structures are silting.
\begin{Assu}\label{Assu:main}
Keep Assumption \ref{Assu}. Assume that $\T$ is a $k$-linear triangulated category and  we have a \emph{relative Serre functor $\SSS$}, 
that is, there is  an auto-equivalence  $\SSS:\T\simeq \T$ which restricts to  an auto-equivalence $\SSS: \U \simeq \U$, such that there exists a functorial isomorphism  
 \[D\Hom_{\T}(X, Y) \cong \Hom_{\T}(Y, \SSS X),\]
 for any $X\in \U$ and $Y\in \T$, where $D$ is the $k$-dual.
\end{Assu}
Then there is a poset isomorphism  as follows.
\begin{Thm}[Theorem \ref{Thm:KY}]\label{Thm:KYA}
Under Assumption \ref{Assu:main},
let  $\T=\X\perp\X^{\perp}={}^{\perp}\cal Z\perp \cal Z$
be any two  $t$-structures.  There is a poset isomorphism  
\[ 
\{P\in \silt\U\mid P\in \X\cap \cal Z \} \xleftrightarrow{\cong}  \{ \H\in\silheart \T\mid \H\subset \X\cap \SSS\cal Z\}.\]
\end{Thm}

Now we consider a finite-dimensional $k$-algebra $A$, which is Iwanaga-Gorenstein  (that is, the $A$-module $A$ has finite injective dimension both sides). For $d\ge 1$, we get the corollary below by applying Theorem \ref{Thm:KY} to 
$\T=\D^{\rm b}(A)$, $\U=\per A$, $\X=\D^{\le 0}$ and $\cal Z=\D^{\ge 1-d}$. It plays a key role in proving Theorem \ref{Thm:bij}.

\begin{Cor}[Corollary \ref{Cor:Gor}]\label{Cor:GorA}
 There is a poset isomorphism
\[(\silt \per A)\cap \D^{[1-d, 0]}\xleftrightarrow{\cong} (\SMC \D^{\rm b}(A))\cap \D^{[-d,0]}_{-}.\]
\end{Cor}





\section{Preliminaries} \label{Section:pre}
Let $\cal T$ be a   triangulated category. Let $\cal U$ and $\cal V$ be two full subcategories of $\cal T$. We denote by $\add \cal U$ the smallest subcategory containing $\cal U$, which is closed under direct summands and finite direct sums. We denote by $\thick\cal U$ the thick subcategory generated by $\cal U$ (that is, the smallest triangulated subcategory of $\T$ containing $\cal U$ and closed under direct summands).
  We denote by $\Filt\,\cal U$ the smallest extension-closed subcategory of $\T$ containing $\cal U$.
We define  new subcategories 
 \begin{eqnarray*}
 \cal U^{\perp} & :=& \{ X\in\T\mid \Hom_{\T}(U, X)=0 \text{ for any } U\in \cal U\},\\ 
 \cal {}^{\perp}\cal U &:=& \{ Y\in \T\mid \Hom_{\T}(Y, U)=0  \text{ for any } U\in \cal U \}, \\
 \cal U \ast \cal V &:= & \{ X\in \cal T \mid \text{there is a triangle $U\ra X \ra V \ra U[1]$ with $U\in \cal U$ and $V\in\cal V$} \}.
 \end{eqnarray*}
  If $\Hom_{\T}(U,V)=0$ for any $U\in \cal U$ and $V\in\cal V$, we write $\cal U\ast\cal V=\cal U\perp \cal V$.

\begin{Def} \label{Def:t-cot}
 Let $\T, \U, \V$ be as above. 
 \begin{enumerate}[\rm(1)]
\item We call $\cal T=\cal U\perp \cal V$ a \emph{torsion pair} of $\T$, if $\cal T=\cal U\perp \cal V$, $\cal U^{\perp}=\cal V$ and ${}^{\perp}\cal V=\cal U$;
\item We call a torsion pair $\cal T=\cal U\perp\cal V$ a \emph{$t$-structure} of $\T$ if 
$\cal U[1]\subset \cal U$;
\item We call a torsion pair $\cal T=\cal U\perp\cal V$ a \emph{co-$t$-structure} of $\T$  if $\cal V[1]\subset \cal V$.
\end{enumerate}
     \end{Def}
 Let $\T=\U\perp\V$ be a $t$-structure (resp. co-$t$-structure) of $\T$.  We denote by $\cal H=\cal U\cap\cal V[1]$ (resp. $\cal P=\cal U\cap \cal V[-1]$) the \emph{heart} (resp. \emph{co-heart}). 
We say a $t$-structure $\T=\cal U\perp \cal V$ is \emph{bounded}, if $\bigcup_{i\in\Z}\U[i]=\T=\bigcup_{i\in\Z}\cal V[i]$. 
A bounded $t$-structure is determined by its heart.
\begin{Lem}\label{Lem:bdd}
Let $\T=\cal U\perp \cal V$ be a  bounded $t$-structure with heart $\H$. Then $\U=\Filt(\H[\ge \hs0])$ and $\V=\Filt(\H[<\hs 0])$.
\end{Lem}

On the set of $t$-structures on $\T$, there is a natural partial order defined by
\[  (\U, \cal V)\ge (\U', \cal V'):\Leftrightarrow \U\supset \U' \Leftrightarrow  \cal V\subset\cal V',\]
where $\T=\U\perp \cal V=\U' \perp\cal V'$ are $t$-structures with hearts $\H$ and $\H'$ respectively. It induces a partial order on the set of hearts of bounded $t$-structures by Lemma \ref{Lem:bdd}, that is 
\begin{equation}  \H \ge \H':  \Leftrightarrow   \U \supset \U' \Leftrightarrow \V\subset\V'\Leftrightarrow\Hom_{\T}(\H', \H[<\hs0])=0. \label{t-order}\end{equation}

\begin{Def} \label{Def:silting}
An object $P\in \T$ is called \emph{silting object} if 
$\Hom_{\T}(P, P[>\hs 0])=0$ and $\T=\thick P$.
\end{Def}

Two silting objects $P$ and $Q$ are \emph{equivalent} if $\add P=\add Q$. 
We denote by $\silt \T$ the set of equivalence classes of silting objects in $\T$. 
If $P\in\T$ is silting, then we have a  natural co-$t$-structure 
\begin{equation} 
 \T=\T^{P}_{\ge0}\perp\T^{P}_{<0}, \  \mbox{where $\T^{P}_{\ge 0}:=\Filt (P[\le \hs 0])$ and $\T^{P}_{<0}:=\Filt (P[>\hs 0])$.} \label{**} \end{equation} 
We have a partial order on $\silt \T$, that is, for $P, Q\in\silt \T$,
\begin{equation}
Q \ge P : \Leftrightarrow \T^{Q}_{<0}\supset\T_{<0}^{P}
\Leftrightarrow \T^{Q}_{\ge 0} \subset \T^{P}_{\ge 0} \Leftrightarrow \Hom_{\T}(Q, P[>\hs 0])=0. \label{co-t-order}
\end{equation}

Let $k$ be a field in the sequel.

\begin{Def} \label{Def:CYtri}
Let $\T$ be a $k$-linear triangulated category. Let $\SSS:\T\xra{\simeq} \T$ be an  equivalence. 
\begin{enumerate}[\rm (1)]
\item  We call $\SSS$  a \emph{Serre functor} of $\T$, if 
there exists a functorial isomorphism for any $X, Y\in \T$,
\[ D\Hom_{\T}(X,Y) \xra{\simeq} \Hom_{\T}(Y, \SSS X). \]

\item
Let $d\in\Z$. We call $\T$ a \emph{$d$-Calabi-Yau triangulated category}, if $[d]$ gives a Serre functor.
\end{enumerate}
\end{Def}

\begin{Def} \label{Def:csilt}
Let $\T$ be a $k$-linear Hom-finite Krull-Schmidt  triangulated category. For $d\ge 1$,  one object $P\in \T$ is called \emph{cluster tilting}, if the following conditions hold.
\begin{enumerate}[\rm  (1)]
\item $\Hom_{\T}(P, P[j])=0$ for $1\le j\le d-1$.
\item We have $\add P =\{ X\in \T \mid \Hom_{\T}(P, X[j])=0 \mbox{ for } 1\le j\le d-1\}$.
\end{enumerate}
\end{Def}

Next let us recall the notions  of SMC and SMS, which are main study objects in our paper. 

\begin{Def}\label{Def:SMC}
Let $\T$ be a $k$-linear Hom-finite Krull-Schmidt  triangulated category and 
 $S$  a set of objects  of $\T$. We say $S$ is a \emph{simple-minded collection} (or \emph{SMC}), if the following conditions hold.
\begin{enumerate}[\rm(1)]
\item  $\End_{\T}(X)$ is a division $k$-algebra for each $X\in \T$, and $\Hom_{\T}(X,Y)=0$ for $X\not=Y\in S$;
\item $\Hom_{\T}(X[>\hs 0], Y)=0$ for $X, Y\in S$;
\item $\T=\thick S$.
\end{enumerate}
\end{Def}
We denote by $\SMC \T$ the set of SMCs in $\T$. 
SMCs were first studied by \cite{Ric} in the context of derived categories of symmetric algebras and they were also studied by \cite{Al} under the name `cohomologically Schurian set of generators'.  The name `SMC' was introduced by \cite[Definition 3.2]{KY} in general setting.

\begin{Def}\label{Def:CY}
Let $\T$ be a $k$-linear Hom-finite Krull-Schmidt  triangulated category and  $S$  a set of 
 objects of $\T$. For $d\ge1$, we call $S$ a \emph{$d$-simple-minded system} (or \emph{$d$-SMS} for short),  if the following conditions hold.
 \begin{enumerate} [\rm (1)]
 \item   $\End_{\T}(X)$ is a division $k$-algebra for each $X\in \T$, and $\Hom_{\T}(X,Y)=0$ for $X\not=Y\in S$;
\item $\Hom_{\T}(X[j], Y)=0$ for any two objects $X, Y$ in $ S$ and $1\le j \le d-1$;

\item 
We have $\T=\Filt (S[-j]\mid 0\le j\le d-1)$.
\end{enumerate}
We call $S$ a \emph{$(-d)$-Calabi-Yau configuration} (\emph{$(-d)$-CY configuration}) if it satisfies (1), (2)  and  $\bigcap\limits_{j=0}^{d-1}{}^{\perp}S[-j]=0$. 
\end{Def}

The notion of SMS was introduced by \cite[Definition 2.1]{Yuming}, and was generalized to $d$-SMS by \cite{CS3}. The term `$(-d)$-CY configuration' studied in \cite{J1} is also referenced to `right $d$-Riedtmann configuration' in \cite{CS} for the same concept.

 
If $H$ is a hereditary $k$-algebra of  Dynkin type,  then  the notion $(-d)$-CY configuration of $\C_{-d}(H)$ coincides with $d$-SMS by \cite[Proposition 2.13]{CSP}. Because in this case, $\Filt(S)$ is functorially finite in $\C_{-d}(H)$ for any $(-d)$-CY configuration $S$.

Finally, we recall the following definition.
\begin{Def}\cite[Definition 2.2]{BRT2}\label{Def:BRT}
Let $H$ be a hereditary $k$-algebra. 
A basic object $X\in \D^{\rm b}(H)$ is a \emph{$\Hom_{\le 0}$-configuration} if 
\begin{enumerate}[\rm(1)]
\item[(H1)] $X$ is the direct sum of $n$ exceptional indecomposable summands $X_{1}, \dots, X_{n}$, where $n$ is the number of simple modules of $H$.
\item[(H2)] $\Hom_{\D}(X_{i},X_{j})=0$ for $i\not=j$.
\item[(H3)] $\Hom_{\D}(X, X[t])=0$ for all $t<0$.
\item[(H4)] There is no subset $\{Y_{1},\dots, Y_{r}\}$ of the indecomposable summands of $X$ such that $\Hom_{\D}(Y_{i}, Y_{i+1}[1])\not=0$ and $\Hom_{\D}(Y_{r}, Y_{1}[1])\not=0$.
\end{enumerate}
\end{Def}

\section{Proof of  main Theorems}

\subsection{Silting-$t$-structure correspondence}
 In this subsection, we prove the results of Section \ref{section:1.2}.
We first  show Theorem \ref{Thm:siltA}. 
The following observation is useful.

\begin{Lem}\label{Lem:poset}
Under Assumption \ref{Assu}, there is a poset isomorphism  $\silt\U\xleftrightarrow{\cong} \silheart\T$. 
\end{Lem}
\begin{proof}
The map $\silt\U\ra \silheart\T$ is clearly surjective. For $P, Q\in\silt\U$, we have 
\[Q\ge P \stackrel{\eqref{co-t-order}}{\Longleftrightarrow} P[\ge \hs0]\subset Q[<\hs0]^{\perp} \stackrel{\eqref{*}}{\Longleftrightarrow} 
P[\ge \hs 0]\subset \T_{Q}^{\le 0}
\stackrel{\eqref{*}}{\Longleftrightarrow}  \T_{Q}^{>0}\subset \T_{P}^{>0}
\stackrel{\eqref{t-order}}{\Longleftrightarrow}\H_{Q}\ge \H_{P}.\]
Thus the map is a poset isomorphism.
\end{proof}


\begin{Prop} \label{Prop:bij}
Under Assumption \ref{Assu}, let $Q,R\in \silt\U$.
Then there is a poset isomorphism 
\[(\silt \U)\cap \U_{\le 0}^{Q}\cap \U_{\ge 0}^{R} \xleftrightarrow{\cong} \{ \H\in \silheart \T \mid 
\H\subset \T_{Q}^{\le 0}\cap \T_{R}^{\ge 0}\}.  \]
\end{Prop}

\begin{proof}
Let $P\in\silt\U$. Then
\[  P\in \U^{Q}_{\le 0} \cap \U_{\ge 0}^{R} {\stackrel{\eqref{co-t-order}}{\Longleftrightarrow}} Q\ge P\ge R \stackrel{\rm Lem.\, \ref{Lem:poset}}{\Longleftrightarrow}   \cal H_{Q}\ge \cal H_{P} \ge \cal H_{R} \stackrel{\eqref{t-order}}{\Longleftrightarrow}    \cal H_{P}\subset \T_{Q}^{\le 0}\cap \T_{R}^{\ge 0}.
\]
 Thus  the assertion holds.
\end{proof}

Now we are ready to prove Theorem \ref{Thm:siltA}, which is stated as follows.
\begin{Thm}[Theorem \ref{Thm:siltA}] \label{Thm:silt}
 Under Assumption \ref{Assu},
let  $\T=\X\perp\X^{\perp}=\Y\perp \Y^{\perp}$
be  two silting $t$-structures.  Then there is a poset isomorphism
\[ 
\{P\in \silt\U\mid P\in \X\cap {}^{\perp}\Y \} \xleftrightarrow{\cong}  \{ \H \in\silheart \T\mid \H \subset \X\cap \Y^{\perp}\}.\]
 \end{Thm}

\begin{proof}
There exists $Q, R\in \silt \U$ such that $\X=Q[<\hs 0]^{\perp}$ and $\Y=R[\le \hs0]^{\perp}$.
Since $\X\cap \Y^{\perp}=\T_{Q}^{\le 0}\cap\T_{R}^{\ge 0}$ and $\X\cap{}^{\perp}\Y\cap\U= \U_{\le 0}^{Q}\cap \U_{\ge 0}^{R}$ hold, the assertion follows from Proposition \ref{Prop:bij}.
\end{proof}

\begin{Cor}[Corollary \ref{Cor:d-termA}]\label{Cor:d-term}
  There is a poset isomorphism.
\[ \{\text{$d$-term silting objects in $\per A$}\}\xleftrightarrow{\cong} (\SMC\D^{\rm b}(A))\cap \D^{[1-d, 0]}.
\]
\end{Cor}

\begin{proof}Let $\T=\D^{\rm b}(A)$ and $\U=\per A$.
Let $\X=\D^{\le 0}$ and $\Y=\D^{\le -d}$. 
Then $\Y^{\perp}=\D^{\ge 1-d}$ and $\per A\cap {}^{\perp}\Y=\Filt (A[<\hs d])$.
By Theorem \ref{Thm:silt}, we have a poset isomorphism  
$ 
\{P\in \per A\mid P\in \D^{\le 0}\cap \Filt A[<\hs d] \} \xleftrightarrow{\cong}  \{ \H\in\silheart \D^{\rm b}(A)\mid \H\subset \D^{[1-d, 0]}\}$.
Using the bijection $\silheart\D^{\rm b}(A)\xleftrightarrow{\cong}\SMC \D^{\rm b}(A)$ in Theorem \ref{Thm:KYbij},
we obtain the assertion.
\end{proof}

Next we prove Theorem \ref{Thm:KYA}.

 


\begin{Thm}[Theorem \ref{Thm:KYA}]\label{Thm:KY}
Under Assumption \ref{Assu:main},
let  $\T=\X\perp\X^{\perp}={}^{\perp}\cal Z\perp \cal Z$
be any two  $t$-structures.  There is a poset isomorphism  
\[ 
\{P\in \silt\U\mid P\in \X\cap \cal Z \} \xleftrightarrow{\cong}  \{ \H\in\silheart \T\mid \H\subset \X\cap \SSS\cal Z\}.\]
\end{Thm}
\begin{proof}
Let $P\in \silt\U$. 
Thanks to Lemma \ref{Lem:poset}, it suffices to show  that  $P\in\X$ if and only if $\cal H_{P}\subset \X$, and 
 $P\in\cal Z$ if and only if $\cal H_{P}\subset \SSS\cal Z$.

\noindent
 (a) By \eqref{*}  and \eqref{**},  we have ${}^{\perp}({\U^{P}_{\le 0}}{}^{\perp})={}^{\perp}\T_{P}^{>0}=\T_{P}^{\le 0}$. Thus 
  \[P\in \cal X\Longleftrightarrow  \U^{P}_{\le 0} \subset \X \stackrel{{}^{\perp}(\X^{\perp})=\X}{\Longleftrightarrow}   {}^{\perp}({\U^{P}_{\le 0}}{}^{\perp}) \subset \X\Longleftrightarrow \T_{P}^{\le 0} \subset \X
\Longleftrightarrow   \cal H_{P}\subset \X.\]
 
\noindent (b)  By \eqref{*}  and \eqref{**},  we have 
$({}^{\perp}\SSS\,\U^{P}_{\ge 0})^{\perp}=({\U^{P}_{\ge 0}}{}^{\perp})^{\perp}={\T^{<0}_{P}}{}^{\perp}=\T_{P}^{\ge 0}$. Thus
    \[ P\in \cal Z \Longleftrightarrow  \U_{\ge 0}^{P} \subset \cal Z \Longleftrightarrow  \SSS \,\U_{\ge 0}^{P} \subset \SSS \cal Z \stackrel{({}^{\perp}\SSS\cal Z)^{\perp}=\SSS\cal Z}{\Longleftrightarrow}  ({}^{\perp} \SSS \,\U_{\ge 0}^{P})^{\perp} \subset \SSS\cal Z\Longleftrightarrow \T_{P}^{\ge 0}\subset \SSS  \cal Z
              \Longleftrightarrow  \cal H_{P}\subset \SSS \cal Z.\] 
   So the assertion is true.
   \end{proof} 
   
\begin{Cor}[Corollary \ref{Cor:GorA}]\label{Cor:Gor}
 There is a poset isomorphism
\[(\silt \per A)\cap \D^{[1-d, 0]}\xleftrightarrow{\cong} (\SMC \D^{\rm b}(A))\cap \D^{[-d,0]}_{-}.\]
\end{Cor}

\begin{proof}
Let $\T=\D^{\rm b}(A)$ and  $\U=\per A$.
Let $\X=\D^{\le 0}$ and $\cal Z=\D^{\ge 1-d}$.
Since $A$ is Iwanaga-Gorenstein, then the Nakayama functor $\nu$ is the relative Serre functor. 
By Theorem \ref{Thm:KY}, we have a poset isomorphism 
\[(\silt \per A )\cap \D^{[1-d,0]}\xleftrightarrow{\cong} \{ \H\in\silheart \D^{\rm b}(A)\mid \H\subset \D^{[-d,0]}_{-}\}.\]
By Theorem \ref{Thm:KYbij}, we obtain the assertion.
\end{proof}

\subsection{Proof of results in Section \ref{section:1.1}}

In this subsection, $H$ is a hereditary $k$-algebra. We will write $\D$ (resp. $\C$) for $\D^{\rm b}(H)$ (resp. $\C_{-d}(H)$) for simplicity, and $\H=\mod H$.
Let $ S$ be an SMC of $\D$. Then $\cal H_{S}:=\Filt S$ is the heart of a  $t$-structure $(\D^{\le 0}_{ S}, \D^{\ge 0}_{S})$ given by 
\[ \D^{\le 0}_{S}:= \Filt(S[\ge \hs 0]) \ \mbox{and} \
\D^{\ge 0}_{S}:=\Filt(S[\le \hs0]). \] 

We need the following  observation. 
\begin{Lem}\label{Lem:1}
Let $S, T$ be two SMCs of $\D$.  Then the following are equivalent.
\begin{enumerate}[\rm(1)]
\item
$\cal H_{S}\subset \D_{T}^{\le 0}\cap \nu \D_{T}^{\ge 1-d}$;
\item
$\cal H_{T}\subset \nu^{-1}\D_{S}^{\le d-1}\cap \D_{S}^{\ge 0}$;
\item
$\D_{S}^{\le 0}\subset \D_{T}^{\le 0}$ and 
$\D^{\ge 0}_{S}\subset \nu \D_{T}^{\ge 1-d}$;
\item 
$\D^{\le 0}_{T}\subset \nu^{-1} \D_{S}^{\le d-1}$ and 
$\D_{T}^{\ge 0}\subset \D_{S}^{\ge 0}$.
\end{enumerate}
\end{Lem}
\begin{proof}
$(1) \Leftrightarrow (3)$  follows from 
\[ (\cal H_{S}\subset \D_{T}^{\le 0} \Longleftrightarrow \D_{S}^{\le 0}\subset \D_{T}^{\le 0})  \mbox{ and } (\cal H_{S}\subset \nu \D_{T}^{\ge 1-d} \Longleftrightarrow \D^{\ge 0}_{S}\subset \nu \D_{T}^{\ge 1-d}).\]
The similar argument shows $(2) \Leftrightarrow (4)$. Finally, 
 $(3)\Leftrightarrow (4)$  follows from 
 \[ (\D_{S}^{\le 0}\subset \D_{T}^{\le 0} \Longleftrightarrow \D_{T}^{\ge 0}\subset \D_{S}^{\ge 0}) \mbox{ and } (\D^{\ge 0}_{S}\subset \nu \D_{T}^{\ge 1-d}  \Longleftrightarrow \D^{\le 0}_{T}\subset \nu^{-1} \D_{S}^{\le d-1}). \qedhere\] 
\end{proof}

Recall that $\pi: \D \ra \C$ is the natural functor and by \eqref{equation:ind-}, there is a bijection
\begin{equation} \ind \D^{[-d,0]}_{-} \xlongrightarrow{\simeq} \ind \C. \label{ind}\end{equation}
In the rest, we write $\pi(X)$ as $X$ for any $X\in \D$.
We give a lemma which plays an important role in the sequel.
\begin{Lem}\label{Lem:hom}
Let $X, Y\in \D^{[-d,0]}_{-} $ and $0\le i\le d$. Then we have 
\[ \Hom_{\C}(X,Y[-i])= \Hom_{\D}(X,Y[-i])\op D\Hom_{\D}(Y, X[i-d]).\]
\end{Lem}

\begin{proof}
By the definition of  $\C$, we have 
\[ \Hom_{\C}(X,Y[-i])=\bigoplus_{n\in\Z}  \Hom_{\D}(X, \nu^{n}Y[nd-i]).\]
If $n<0$, then $nd-i\le -d$ and $\nu^{n}Y[nd-i]\in \nu^{n+1}\D^{\ge 1}\subset \D^{\ge 1}$ by \eqref{nu}. So  $\Hom_{\D}(X, \nu^{n}Y[nd-i])=0$.
If $n>1$, then 
$m:=1-n<0$ and 
\[ \Hom_{\D}(X, \nu^{n}Y[nd-i])=D \Hom_{\D}(\nu^{n-1}Y[nd-i], X)=D\Hom_{\D}(Y, \nu^{m}X[md-(d-i)])=0,\]
by the first case.
Thus the assertion follows.
\end{proof}


For $m\ge n$, we denote by $\D_{S}^{[m,n]}$ the intersection $\D_{S}^{\le n}\cap \D_{S}^{\ge m}$. The following lemma is useful.
\begin{Lem}\label{Lem:perp}
Let $H$ be a hereditary $k$-algebra,  $S\in \SMC\D\cap \D^{\le 0}\cap \nu\D^{\ge 1-d}$ and  $N\in\D^{\le 0}\cap \nu\D^{\ge 1-d}$. For $1\le i\le d$, if $\Hom_{\D}(S[i],N)=0=\Hom_{\D}(N, S[i-d])$, then $N=0$.
\end{Lem}
\begin{proof}
By \eqref{nu} and Lemma \ref{Lem:1}, we have 
\begin{equation} \label{mod}
\H=\mod A \subset \nu \D^{\le 1}\cap \D^{\ge 0} \subset \D^{\le d}_{S}\cap \D_{S}^{\ge 0}. 
\end{equation}
Since $DA\in \nu\D^{\le 0} \subset \D^{\le d-1}_{S}$  by   Lemma \ref{Lem:1},  
 $DA\in \D_{S}^{[0,d-1]}$ by \eqref{mod}. 
Note that $\Hom_{\D}(N, S[i-d])=0$ for $1\le i\le d$, 
then $\Hom_{\D}(N, X)=0$ for any $X\in \D_{S}^{[0, d-1]}$. In particular,
 we have $\Hom_{\D}(N, DA)=0$ (that is $\h^{0}(N)=0$), and moreover, $N\in \D^{\le -1}\cap \nu \D^{\ge 1-d}\subset \D^{[-d,-1]}$.
 Since $\D^{[-d,-1]}=\Filt \{\H[i]\mid 1\le i\le d\}$, then \eqref{mod} implies that $N\in \D^{[-d,-1]}\subset \D_{S}^{[-d, d-1]}$.
 Recall that $\D_{S}^{\le d-1}={}^{\perp}(S[\le \hs -d])$ and $\D_{S}^{\ge -d}=(S[>\hs d])^{\perp}$, thus $N \in  \D^{\le -1}_{S} \cap \D^{\ge 0}_{S}=0$ by the orthogonality conditions on $N$.
\end{proof}

We denote by $(-d)\text{-}\CYconf\C$ the set of $(-d)$-Calabi-Yau configurations of $\C$.

\begin{Prop}\label{Prop:well-defined}
Let $H$ be a hereditary $k$-algebra. Then the map  
\begin{equation} \label{pi}
(\SMC\D)\cap \D^{[-d,0]}_{-} \xlongrightarrow{\pi}
(-d)\text{-}\CYconf\C.\end{equation}
is well-defined.
\end{Prop}

\begin{proof}
 Let $S$ be an SMC contained in $\D^{[-d,0]}$. 
 We show $S$ is a $(-d)$-CY configuration in $\C$. Let $X, Y\in S$ and  $0\le i<d$. By Lemma \ref{Lem:hom}, we have 
 \begin{align*} \Hom_{\C}(X, Y[-i])=\Hom_{\D}(X, Y[-i])\op D\Hom_{\D}(Y, X[i-d]). \end{align*} 
Immediately  $S$ satisfies  the conditions (1) and (2) in Definition \ref{Def:CY}. 

 It remains to check that $\bigcap\limits_{j=0}^{d-1}{}^{\perp}S[-j]=0$. 
Let $M\in \C$ be an indecomposable object satisfying  $\Hom_{\C}(M, S[-j])=0$ for any $0\le j\le d-1$. By \eqref{ind}, there exists an indecomposable object $N\in \D^{[-d,0]}_{-}$, such that $\pi(N)=M$. By   Lemma \ref{Lem:hom}, 
 we have   $\Hom_{\D}(N, S[i-d])=0$ and $\Hom_{\D}(S[i], N)=0$ for any $1\le i\le d$. Then $N=0$ by Lemma \ref{Lem:perp}.
\end{proof}

We are ready to prove Theorem \ref{Thm:bijA}, which is stated as follows.

\begin{Thm}\label{Thm:bij}
Let $H$ be a hereditary $k$-algebra of Dynkin type and $d\ge 1$. Then there are  bijections
\begin{eqnarray}
(\silt \cal D^{\rm b}(H))\cap \cal D^{[1-d,0]} 
&\xleftrightarrow{1:1} &
(\SMC\D^{\rm b}(H))\cap \D^{[-d, 0]}_{-}, \label{silttoSMCA}
\\
&\xleftrightarrow{1:1}& 
\SMS\cal C_{-d}(H), \label{SMCtoSMSA}
\end{eqnarray}
where the map \eqref{SMCtoSMSA} is induced by the natural functor $\cal D^{\rm b}(H) \xra{\pi} \cal D^{\rm b}(H)/\nu[d].$

\end{Thm}

\begin{proof}
The bijection \eqref{silttoSMCA} follows directly from Corollary \ref{Cor:Gor}. The  map  \eqref{SMCtoSMSA} is well-defined by Proposition \ref{Prop:well-defined}.
Since this is injective by \eqref{ind}, it suffices to show that \eqref{SMCtoSMSA} is surjective.

 Let $S$ be any $d$-SMS of $\C$. 
We also denote by $S$ the preimage $\pi^{-1}(S)$
of $S$ via the bijection \eqref{ind}.
We claim $S$ is an SMC of $\D$. For $X\in S$, by Lemma \ref{Lem:hom} and $\Hom_{\D}(X, X[<\hs 0])=0$,  we have 
$ \End_{\cal C}(X)=\End_{\D}(X)$,
thus $\End_{\D}(X)$ is a division ring. Let $X,Y\in S$.
 We have  
\[\Hom_{\cal C}(X, Y[-i])\supset \Hom_{\D}(X, Y[-i]),\]
 for any $i\in \Z$. 
For $0\le i\le d-1$ and $X\not=Y$, the left hand side is $0$, so is the right hand side.
 Now we show $\Hom_{\D}(X, Y[-d])=0$. This is clear if $X=Y$.
 If $X\not =Y$, then by Lemma \ref{Lem:hom},
\[ 0=\Hom_{\C}(Y,X) = \Hom_{\D}(Y,X)\op D\Hom_{\D}(X, Y[-d]), \]
and hence $\Hom_{\D}(X, Y[-d])=0$.
For $i>d$,  we have
 \[Y[-i]\in \nu \D^{\ge1-d+i}\subset \nu \D^{>1}\subset \D^{\ge 1}\]
 by \eqref{nu}. Thus  $\Hom_{\D}(X, Y[<\hs -d])=0$.

 It remains to show  $\D=\thick S$.
Since $\D$ is locally finite,  $\cal S=\thick S$ is functorially finite in $\D$. Thus we have a torsion pair $\D={}^{\perp}\cal S\perp \cal S$ by \cite[Proposition 2.3]{IY1}.
Thus it suffices to show  ${}^{\perp}\cal S=0$. Let  $X\in {}^{\perp}\cal S$ be an indecomposable object. 
Since $H$ is hereditary, $\D=\add(\cal H[i]\mid i\in \Z)$.
We may assume $X\in\cal H$. 
Then $\Hom_{\D}(X, S[-i])=0$ for all $i\in\Z$. Moreover, for any $0\le i<d$, we have $X[i-d]\in \D^{\ge 1}$, and hence 
$\Hom_{\D}(S,X[i-d])=0$.   Since $X, S\in \D^{[-d,0]}_{-}$, we have $\Hom_{\C}(X, S[-i])=0$  by Lemma \ref{Lem:hom}. Since $S$ is a $d$-SMS,   $X=0$.
Thus ${}^{\perp}\cal S=0$ as desired.
\end{proof}


Theorem \ref{Thm:FCA} is clear now, which is stated as follows.

\begin{Thm}\label{Thm:FC}
Let $H$ be a hereditary $k$-algebra of  Dynkin type  and let $d\ge 1$.
\begin{enumerate}[\rm (1)]
 \item There is a bijection 
 \[ \SMS\cal C_{-d}(H)\xleftrightarrow{1:1} \pmaxsim\Delta^{d}(\Phi).\]
 \item
We have 
$\#\SMS\cal C_{-d}(H)=C_{d}^{+}(W)$,
where $W$ is the Weyl group of $H$. 
\end{enumerate}
    \end{Thm}

\begin{proof}
(1) 
By \cite[Theorem 5.7]{Z} and \cite[Proposition 2.4]{BRT2}, we have bijections
\[     \maxsim\Delta^{d} (\Phi) 
\xra{\sim}   \ctilt \C_{d+1}(H)  \xra{\sim}  (\silt \D^{\rm b}(H)) \cap \D^{[1-d,0]}_{+},
 \] 
 which restrict to a bijection
 \[\pmaxsim \Delta^{d} (\Phi)  \xra{\sim} \ctilt^{+} \C_{d+1}(H)
 \xra{\sim} (\silt \D^{\rm b}(H))\cap \D^{[1-d,0]},  \]
 where $\ctilt^{+} \C_{d+1}(H)$ consists of $P\in \ctilt \C_{d+1}(H)$, which does not have a non-zero common direct summand with $(\proj H)[d]$.
 Combine with  Theorem \ref{Thm:bij}, we get a bijection
 $
\pmaxsim\Delta^{d}(\Phi) \xra{\sim}
\SMS\cal C_{-d}(H).
$
Then (1) is true.
 
(2) By \cite[Proposition 12.4]{FR},  $\#\pmaxsim\Delta^{d}(\Phi)=C_{d}^{+}(W)$. Then (2) follows by (1).
\end{proof}

Let $H$ be a hereditary $k$-algebra. 
Recall that \cite{BRT} introduced the notion of 
\emph{$\Hom_{\le 0}$-configurations} of $\D^{\rm b}(H)$ {(see Definition \ref{Def:BRT})}, and  they gave a bijection between silting objects and $\rm{Hom}_{\le 0}$-configurations in \cite[Theorem 2.4]{BRT}.  Combining their 
result and a general result Theorem \ref{Thm:appendix2} on SMC-exceptional sequence,  we prove the following result.

\begin{Thm}\label{Thm:appendixA}
Let $H$ be a hereditary $k$-algebra. Then $\rm{Hom}_{\le 0}$-configurations of $\D^{\rm b}(H)$ are precisely SMCs  of $\D^{\rm b}({H})$.
\end{Thm}

To prove this theorem, we need the following observation.
\begin{Prop} \label{Prop:appendix}
Let $H$ be a hereditary algebra. Then 
\begin{enumerate}[\rm (1)]
\item For each $X_{1}\op\cdots\op X_{n} \in \silt {\D^{\rm b}(H)}$, the sequence $(X_{1}, \dots, X_{n})$ can be ordered  into a silting-exceptional sequence.
\item For each $\{S_{1}, \dots, S_{n}\} \in \SMC{\D^{\rm b}(H)}$, the sequence $(S_{1}, \dots, S_{n})$ can be ordered into a SMC-exceptional sequence.
 \end{enumerate}
\end{Prop}
\begin{proof}
(1) See \cite[Proposition 3.11]{AI}.

(2) Let $S=\{S_{1} , \dots,  S_{n}\}\in \SMC\D^{\rm b}(H)$. By Theorem \ref{Thm:KYbij}, there exists a silting object $P=\bop_{i=1}^{n}P_{i}$, such that $S$ is the set of simple objects in the heart of $t$-structure $\D^{\rm b}(H)=P[<\hs0]^{\perp}\perp P[\ge \hs 0]^{\perp}$. By (1), we may reorder $(P_{1}, \dots, P_{n})$ into a silting-exceptional sequence $(P_{l_{1}}, \dots, P_{l_{n}})$. Then by Theorem \ref{Thm:appendix2}, $(P_{l_{n}}^{\vee}, \dots, P_{l_{1}}^{\vee})=\mu^{+}_{\rev}(P_{l_{1}}, \dots, P_{l_{n}})$ is a SMC-exceptional sequence. Moreover, it also gives a set of simple objects in the heart of $\D^{\rm b}(H)=P[<\hs0]^{\perp}\perp P[\ge \hs 0]^{\perp}$ by Lemma \ref{Lem:appendix}(2). Therefore, $(P_{l_{n}}^{\vee}, \dots, P_{l_{1}}^{\vee})$ can be obtained by reordering $S$.
\end{proof}

\begin{proof}[Proof of Theorem \ref{Thm:appendixA}]
Notice that $\mu_{\rev}^{+}$ induces bijections
\[\xymatrix{ {\cal A:=}\{\mbox{SMC-exceptional sequences of $\D^{\rm b}(H)$}\} \\
\{\mbox{silting-exceptional sequences of $\D^{\rm b}(H)$}\} \ar[d]^{\mbox{\cite[Theorem 2.4]{BRT}}}_{\mu_{\rev}^{+}} \ar[u]^{\mu_{\rev}^{+}}_{\mbox{Theorem \ref{Thm:appendix2}}}
\\
 {\cal B:=}\{\mbox{exceptional sequences of $\D^{\rm b}(H)$ which are $\Hom_{\le 0}$-configurations}\}.
}\]
Then $\cal A=\cal B$. Thus, by Proposition \ref{Prop:appendix}(2) (resp. \cite[Lemma 2.3]{BRT}), any SMC (resp. $\Hom_{\le 0}$-configuration) of $\D^{\rm b}(H)$ is given by an object in $\cal A$ (resp. $\cal B$) by forgetting the order.
So the assertion is true.
\end{proof}

\appendix
\section{Silting-$t$-structure correspondence via exceptional mutation}

In this appendix, we study the correspondence between silting objects and SMCs by exceptional mutation \cite{CB, GR}. Throughout this section, $k$ is a field.
Let $\T$ be a $k$-linear triangulated category such that $\bop_{n\in\Z}\Hom_{\T}(X, Y[n])$ is finite-dimensional for any $X, Y\in \T$. An object $X\in \T$ is \emph{exceptional} if $\End_{\T}(X)$
is a division ring and $\Hom_{\T}(X, X[i])=0$ for any $i\not=0$. We call a sequence $(X_{1}, \dots, X_{n})$ of exceptional objects of $\T$ an \emph{exceptional sequence} if $\Hom_{\T}(X_{i}, X_{j}[l])=0$ for any $1\le j < i\le n$ and any $l \in \Z$. An exceptional sequence $(X_{1},\dots, X_{n})$ is \emph{full} if $\thick_{\T}(X_{1}\op\cdots\op X_{n})=\T$. If $\bop_{i=1}^{n}X_{i}$ is a silting object (resp. SMC) of $\T$, we say the exceptional sequence $(X_{1}, \dots, X_{n})$ is  \emph{silting-exceptional} (resp. \emph{SMC-exceptional}).

Let $(X, Y)$ be an exceptional sequence of $\T$. Consider the following two triangles
\[
Y'[-1] \ra \bop_{i\in \Z}\Hom_{\T}(X[i], Y)\ot_{k} X[i]\ra Y \ra Y', 
\]
\[ X'\ra  X \ra \bop_{\j\in\Z}D\Hom_{\T}(X, Y[j]) \ot_{k}Y[j] \ra X'[1].
\]
We denote by $\mu^{+}(X, Y):=(Y', X)$ and $\mu^{-}(X, Y):=(Y, X')$ the \emph{left and right mutations of $(X, Y)$} respectively.
For an exceptional sequence $(X_{1}, \dots, X_{n})$ and $1\le i\le n-1$, we define 
\[\mu_{i}^{+}(X_{1},\dots, X_{n}):=(X_{1}, \dots, X_{i-1}, \mu^{+}(X_{i}, X_{i+1}), X_{i+2}, \dots, X_{n}),\]
\[\mu_{i}^{-}(X_{1},\dots, X_{n}):=(X_{1}, \dots, X_{i-1}, \mu^{-}(X_{i}, X_{i+1}), X_{i+2}, \dots, X_{n}).\]
Then $\mu_{i}^{+}(X_{1},\dots, X_{n})$ and $\mu_{i}^{-}(X_{1},\dots, X_{n})$ are also exceptional sequences.  Moreover, $\mu_{i}^{+}$ and $\mu^{-}_{i}$ are mutual inverses, and they satisfy the braid relations.
Let 
\[ \mu^{+}_{\rm rev}:=\mu_{1}^{+}(\mu_{2}^{+}\mu_{1}^{+})\cdots(\mu_{n-1}^{+}\cdots\mu_{1}^{+}) \ \mbox{ and } \ \mu^{-}_{\rm rev}:=\mu_{1}^{-}(\mu_{2}^{-}\mu_{1}^{-})\cdots(\mu_{n-1}^{-}\cdots\mu_{1}^{-}). \]

Let $\fullexc \T$ (resp. $\siltexc \T$, $\SMCexc\T$) be the set of 
isomorphism classes of full exceptional (resp. silting-exceptional, SMC-exceptional) sequences in $\T$.

\begin{Thm}\label{Thm:appendix2}
$\mu^{\pm}_{\rm rev}$ gives bijections
\[\siltexc\T \xra{1:1} \SMCexc\T \ \text{  and  }  \ \SMCexc\T \xra{1:1} \siltexc\T. \]
\end{Thm}

\begin{Rem}
Theorem \ref{Thm:appendix2} is given in  \cite[Theorem 4.6]{BRT} for the derived category $\D^{\rm b}(H)$ of a hereditary algebra $H$, where they considered $\Hom_{\le 0}$-configurations instead of SMCs. In Proposition \ref{Thm:appendixA}, we show $\Hom_{\le 0}$-configurations are precisely SMCs in the derived category of hereditary algebras by using Theorem \ref{Thm:appendix2}.
\end{Rem}

To prove Theorem \ref{Thm:appendix2},
the following observation is useful.
\begin{Lem} \label{Lem:appendix}
Let $(X_{1}, \dots, X_{n})$ be an exceptional sequence and  $(X_{n}^{\vee}, \dots, X_{1}^{\vee}):=\mu^{+}_{\rev}(X_{1}, \dots, X_{n})$. Then 
\begin{enumerate}[\rm(1)]
\item $X_{i}^{\vee}\in X_{i}\ast \add(X_{1}^{\vee}[\Z])\ast\cdots \ast \add (X_{i-1}^{\vee}[\Z])\subset X_{i}\ast \Filt (X_{1}[\Z],\dots,X_{i-1}[\Z])$.

\item For any $1\le s, t\le n$ and $m\in \Z$, we have 
\begin{equation} \label{appendix:KY}
\Hom_{\T}(X_{s}, X_{t}^{\vee}[m])=\begin{cases} \mbox{division ring,}  & \text{if $s=t$ and $n=0.$}  \\ 0, & \text{otherwise.}
\end{cases} 
\end{equation}

\item If $(X_{1},\dots, X_{n})$ is a silting-exceptional sequence, then for each $1\le i\le n$, we have 
\[ X_{i}^{\vee}\in X_{i} \ast \add(X_{1}^{\vee}[>\hs 0])\ast \cdots \ast \add(X_{i-1}^{\vee}[>\hs 0])\subset X_{i}\ast\Filt(X_{1}[>\hs0], \dots, X_{i-1}[>\hs 0]) .  \]

\end{enumerate}
\end{Lem}

\begin{proof}
(1) Let $\mu^{(i)}:=\mu_{n-i}^{+}\cdots\mu_{1}^{+}$. Then $\mu_{\rev}^{+}=\mu^{(n-1)}\cdots\mu^{(1)}$. Let $X_{i}^{(0)}:=X_{i}$, for any $1\le i\le n$, and $(X_{i+1}^{(i)}, \dots, X_{n}^{(i)}, X_{i}^{(i)},\dots, X_{1}^{(i)}):=\mu^{(i)}\cdots\mu^{(1)}(X_{1},\dots, X_{n})$.  
Since \[(X_{i+1}^{(i)}, \dots, X_{n}^{(i)}, X_{i}^{(i)},\dots, X_{1}^{(i)})=\mu^{(i)}(X_{i}^{(i-1)}, \dots, X_{n}^{(i-1)}, X_{i-1}^{(i-1)},\dots, X_{1}^{(i-1)}),\] then we have $X_{j}^{(i)}=X_{j}^{(i-1)}$ for $1\le j\le i$,
and moreover, for any $i+1\le t\le n$, we have 
\begin{equation}\label{equation:Lem}
X_{t}^{(i)}\in X_{t}^{(i-1)}\ast \add (X_{i}^{(i-1)}[\Z]).\end{equation}
Therefore,
\begin{align*} X_{i}^{\vee}=X_{i}^{(n-1)}=\cdots=X_{i}^{(i-1)}  \in   X_{i}^{(i-2)}\ast \add(X_{i-1}^{\vee}[\Z]) &  \subset X_{i}^{(i-3)}\ast \add(X_{i-2}^{\vee}[\Z])\ast \add (X_{i-1}^{\vee}[\Z])  \\
 & \subset \cdots \subset X_{i}\ast \add(X_{1}^{\vee}[\Z])\ast \cdots \ast \add(X_{i-1}^{\vee}[\Z]).\end{align*}
Inductively,  we  have $ X_{i}\ast \add(X_{1}^{\vee}[\Z])\ast \cdots \ast \add(X_{i-1}^{\vee}[\Z])\subset X_{i}\ast \Filt(X_{1}[\Z], \dots, X_{i-1}[\Z])$. So (1) holds. 

(2)  Assume $s\ge t$. Since $X_{t}^{\vee}\in X_{\le t}\ast\Filt(X_{1}[\Z], \cdots, X_{t-1}[\Z])$ by (1) and  $\Hom_{\T}(X_{s}, X_{t}[\Z])=0$,
we obtain
\[ \Hom_{\T}(X_{s}, X_{t}^{\vee}[m])=\Hom_{\T}(X_{s}, X_{t}[m])=\begin{cases} \mbox{division ring,}  & \text{if $s=t$ and $m=0.$}  \\ 0, & \text{otherwise.}
\end{cases} \]
Assume $s<t$.  Then $X_{s}\in \add(X_{1}^{\vee}[\Z])\ast \cdots\ast \add(X_{s-1}^{\vee}[\Z])$ by (1). Since $(X_{n}^{\vee},\dots, X_{1}^{\vee})$  is an exceptional sequence, $\Hom_{\T}(X_{\le s}^{\vee}, X_{t}^{\vee}[\Z])=0$ holds, thus  
$\Hom_{\T}(X_{s}, X_{t}^{\vee}[m])=0. $
So the assertion is true.

(3) As in the proof of (1), it is enough to show that  for any $1\le i\le n$, we have  
\begin{equation}\label{equation:LemA2} 
X_{t}^{(i)}\in X_{t}^{(i-1)}\ast \add(X_{i}^{(i-1)}[>\hs 0])   \mbox{ for $i+1\le t\le n$.}
\end{equation} In fact, $X_{t}^{(i)}$ is given by the triangle
\begin{equation*}\label{equation:Lem2}
X_{t}^{(i)}[-1] \ra \bop_{l\in \Z}\Hom_{\T}(X_{i}^{(i-1)}[l], X_{t}^{(i-1)})\ot_{k} X_{i}^{(i-1)}[l]\ra X_{t}^{(i-1)} \ra X_{t}^{(i)}.
\end{equation*}
To prove \eqref{equation:LemA2},  it suffices to show that 
\begin{equation} \label{equation:Lem3}
\Hom_{\T}(X_{i}^{(i-1)}[l], X_{t}^{(i-1)})=0 \mbox{ for $l\le 0$.}
\end{equation}
This is clearly  for $i=1$. Now we assume \eqref{equation:LemA2} is true for $1, 2, \dots, i-1$.
Since $X_{i}^{(i-1)}=X_{i}^{\vee}\in X_{i}\ast \add(X_{1}^{\vee}[\Z])\ast\cdots\ast\add(X_{i-1}^{\vee}[\Z])$ by (1), and $(X_{i}^{(i-1)}, \dots, X_{n}^{(i-1)}, X_{i-1}^{(i-1)},\dots, X_{1}^{(i-1)})$ is an exceptional sequence, then
 $\Hom_{\T}(X_{j}^{\vee}[\Z], X_{t}^{(i-1)})=\Hom_{\T}(X_{j}^{(i-1)}[\Z], X_{t}^{(i-1)})=0$  for any $1\le j\le i-1$.
Therefore,  we have 
\[\Hom_{\T}(X_{i}^{(i-1)}[l], X_{t}^{(i-1)})=\Hom_{\T}(X_{i}[l], X_{t}^{(i-1)}).  \]
It is zero for $l<0$, because  we know $X_{t}^{(i-1)}\in X_{t}\ast\Filt(X_{1}[>\hs0],\dots,X_{i-1}[>\hs 0])$ by our assumption that \eqref{equation:LemA2} is true for $1,2,\dots, i-1$. 
Thus \eqref{equation:Lem3} holds.
\end{proof}

Now we prove Theorem \ref{Thm:appendix2}.
\begin{proof}[Proof of Theorem \ref{Thm:appendix2}]
For any full exceptional sequence $(X_{1},\dots, X_{n})$, we know $(X_{n}^{\vee},\dots, X_{1}^{\vee}):=\mu(X_{1},\dots, X_{n})$ is also a full exceptional sequence. Therefore we have $\thick_{\T}(X_{1}^{\vee}, \dots, X_{n}^{\vee})=\T$ and $\Hom_{\T}(X_{i}^{\vee}, X_{i}^{\vee}[\not=0])=0$.

Now we show the first part.  Fix $(P_{1}, \dots, P_{n})\in \siltexc\T$. To show $(P_{n}^{\vee}, \dots, P_{1}^{\vee})\in \SMCexc\T$, it remains to show $\Hom_{\T}(P_{s}^{\vee}, P_{t}^{\vee}[\le \hs 0])=0$ for each $s\not=t$.
 By Lemma \ref{Lem:appendix}(3), we have
 $P_{s}^{\vee}\in P_{s}\ast\Filt(P_{1}[>\hs0], \dots, P_{s-1}[>\hs 0])$. By Lemma \ref{Lem:appendix}(2), $\Hom_{\T}(P_{s}, P_{t}^{\vee}[\le \hs 0])=0$ and $\Hom_{\T}(P_{s'}[	>\hs 0], P_{t}^{\vee}[\le \hs 0])=0$ for each $s'$. So $(P_{n}^{\vee}, \dots, P_{1}^{\vee})\in \SMCexc\T$.
 
 On the other hand, for any full exceptional sequence $(X_{1}, \dots, X_{n})$ with $(X_{n}^{\vee},\dots, X_{1}^{\vee})\in \SMCexc\T$, we show $(X_{1}, \dots, X_{n})\in \siltexc\T$. 
 By Lemma \ref{Lem:appendix}(3), we have $X_{t}\in \add(X_{1}^{\vee}[\ge\hs0])\ast\cdots\add(X_{t-1}^{\vee}[\ge \hs 0])\ast X_{t}^{\vee}$  and $\Hom_{\T}(X_{s}, X_{t'}^{\vee}[>\hs0])=0$ for any $s,t'$. Then $\Hom_{\T}(X_{s}, X_{t}[>\hs0])=0$ and $(X_{1}, \dots, X_{n})\in \siltexc\T$. 
 Thus the assertion follows.
 
We have a bijection $\iota: \fullexc \T \xra{1:1} \fullexc \T^{\rm op}$ given by $\iota(X_{1},\dots, X_{n})=(X_{n}, \dots, X_{1})$, which gives bijections  $\iota: \siltexc\T \xra{1:1} \siltexc \T^{\rm op}$ and $\iota: \SMCexc\T \xra{1:1} \SMCexc \T^{\rm op}$.
 Since  $\iota\circ \mu_{i}^{\pm}\circ \iota=\mu_{i}^{\mp}$ and hence, $\iota\circ \mu_{\rev}^{\pm}\circ \iota=\mu_{\rev}^{\mp}$ holds,  then the second claim follows from the first one.
\end{proof}

\begin{Rem}
Similarly to Lemma \ref{Lem:appendix}(3), we can show that, 
 if $(X_{1},\dots, X_{n})$ is a SMC-exceptional sequence, then for each $1\le i\le n$, we have 
\[ X_{i}^{\vee}\in X_{i} \ast \add(X_{1}^{\vee}[\le\hs 0])\ast \cdots \ast \add(X_{i-1}^{\vee}[\le\hs 0])\subset X_{i}\ast\Filt(X_{1}[\le\hs 0], \dots, X_{i-1}[\le\hs 0]).\]
The second part of Theorem \ref{Thm:appendix2} can be shown directly by this result.

\end{Rem}

We give an application of Theorem \ref{Thm:appendix2}.

\begin{Cor}
Let $A$ be a proper dg $k$-algebra. If $\per A$ has a full exceptional sequence, then
 $\per A$  has silting-exceptional and SMC-exceptional  sequences. Moreover, we have $\per A=\D^{\rm b}(A)$.
\end{Cor}

\begin{proof}
Let $(X_{1}, \dots, X_{n})$ be a full exceptional sequence in $\per A$. Then by \cite[Proposition 3.5]{AI}, we have a silting-exceptional sequence $(X_{1}[l_{1}], \dots, X_{n}[l_{n}])$ by shifting the original one with proper integers $l_{1},\dots, l_{n}$. By Theorem \ref{Thm:appendix2}, we get  a SMC-exceptional sequence $\mu_{\rev}^{+}(X_{1}[l_{1}], \dots, X_{n}[l_{n}])$.  Let  $X=\bop_{i=1}^{n}X_{i}[l_{i}]$. 
Then $B:=\shEnd_{A}(X)$  is a non-positive proper dg algebra. Moreover,  we have  a triangle equivalence $\per B\simeq \per A$ given by $\RshHom_{A}(X, ?)$,  and hence $\D^{\rm b}(B)\simeq \D^{\rm b}(A)$. 

Recall that $\D^{\rm b}(B)$ has a standard $t$-structure $\D^{\rm b}(B)=\D^{\le 0}\perp \D^{> 0}$, which is bounded and the heart is a length category.  By Lemma \ref{Lem:appendix}(2), the simple objects in the heart  are given by $\mu_{\rev}^{+}(X_{1}[l_{1}], \dots, X_{n}[l_{n}])  \subset \per B$. Therefore $\D^{\rm b}(B)=\thick(\mu_{\rev}^{+}(X_{1}[l_{1}], \dots, X_{n}[l_{n}]))\subset \per B$. Since $B$ is proper, we have $\D^{\rm b}(B)=\per B$ as desired.
\end{proof}



%

\medskip\noindent
{\bf Acknowledgements }
We thank the referees for their useful comments and suggestions, which improved the writing of our paper. 



\begin{thebibliography}{10}


\bibitem[Al]{Al}
Salah Al-Nofayee,
\emph{Simple objects in the heart of a t-structure}, 
J. Pure Appl. Algebra 213 (2009), no. 1, 54--59. 


\bibitem[Ar]{A}
Drew Armstrong, \emph{Generalized noncrossing partitions and combinatorics of Coxeter groups}, Mem. Amer. Math. Soc. 202 (2009), no. 949. 


\bibitem[AI]{AI}
Takuma Aihara, Osamu Iyama, \emph{Silting mutation in triangulated categories}, J. Lond. Math. Soc. (2) 85 (2012), no. 3, 633--668.



\bibitem[BMRRT]{BMRRT}
Aslak Bakke Buan, Robert Marsh, Markus Reineke, Idun Reiten, Gordana Todorov, \emph{Tilting theory and cluster combinatorics},
Adv. Math. 204 (2006), no. 2, 572--618.


\bibitem[BRT1]{BRT2}
Aslak Bakke Buan, Idun Reiten, Hugh Thomas, \emph{Three kinds of mutation},  
J. Algebra 339 (2011), 97--113.


\bibitem[BRT2]{BRT}
Aslak Bakke Buan, Idun Reiten, Hugh Thomas,
\emph{From m-clusters to m-noncrossing partitions via exceptional sequences},
Math. Z. 271 (2012), no. 3-4, 1117--1139. 



\bibitem[BY]{BY}
Thomas  Br\"ustle, Dong Yang, \emph{Ordered exchange graphs}, Advances in representation theory of algebras, 135--193, EMS Ser. Congr. Rep., Eur. Math. Soc., Z\"urich, 2013. 

\bibitem[C]{CB}
 William Crawley-Boevey, \emph{Exceptional sequences of representations of quivers}, Representations of algebras (Ottawa, ON, 1992), 117--124, CMS Conf. Proc., 14, Amer. Math. Soc., Providence, RI, 1993.
 

\bibitem[CS1]{CS0} Raquel Coelho Sim\~ oes, \emph{Hom-configurations and noncrossing partitions}, J. Algebraic Combin. 35 (2012), no. 2, 313--343. 

\bibitem[CS2]{CS} Raquel Coelho Sim\~ oes, \emph{Hom-configurations in triangulated categories generated by spherical objects}, J. Pure Appl. Algebra 219 (2015), no. 8, 3322--3336. 


\bibitem[CS3]{CS3} Raquel Coelho Sim\~ oes, 
\emph{Mutations of simple-minded systems in Calabi-Yau categories generated by a spherical object}, Forum Math. 29 (2017), no. 5, 1065--1081. 







\bibitem[CSP]{CSP} Raquel Coelho Sim\~oes, David Pauksztello, \emph{Simple-minded systems and reduction for negative Calabi-Yau triangulated categories}, to appear in Trans. Amer. Math. Soc., arXiv:1808.02519.

\bibitem[CSPP]{CSPP} Raquel Coelho Sim\~oes, David Pauksztello, David Ploog, \emph{Functorially finite hearts, simple-minded systems in negative cluster categories, and noncrossing partitions}, arXiv:2004.00604. 

\bibitem[D]{D}
Alex Dugas, \emph{Torsion pairs and simple-minded systems in triangulated categories},
Appl. Categ. Structures 23 (2015), no. 3, 507--526. 

\bibitem[FR]{FR}
Sergey Fomin, Nathan Reading, \emph{Generalized cluster complexes and Coxeter combinatorics}, 
Int. Math. Res. Not. 2005, no. 44, 2709--2757. 


\bibitem[FZ]{FZ}
Sergey Fomin, Andrei Zelevinsky,
\emph{Cluster algebras. II. Finite type classification}, Invent. Math. 154 (2003), no. 1, 63--121.

\bibitem[GR]{GR}
A. L. Gorodentsev, A.N. Rudakov, \emph{Exceptional vector bundles on projective spaces}, Duke Math. J. 54 (1987), no. 1, 115--130.


\bibitem[HJY]{HJY}
Thorsten Holm, Peter J\o rgensen, Dong Yang,
\emph{Sparseness of $t$-structures and negative Calabi-Yau dimension in triangulated categories generated by a spherical object}, Bull. Lond. Math. Soc. 45 (2013), no. 1, 120--130. 



\bibitem[I]{I}
Osamu Iyama, \emph{Tilting Cohen-Macaulay representations}, Proceedings of the International Congress of Mathematicians--Rio de Janeiro 2018. Vol. II. Invited lectures, 125--162, World Sci. Publ., Hackensack, NJ, 2018. 

\bibitem[IYa]{IY}
Osamu Iyama, Dong Yang, \emph{Silting reduction and Calabi-Yau reduction of triangulated categories}, Trans. Amer. Math. Soc. 370 (2018), no. 11, 7861--7898.

\bibitem[IYo]{IY1}
Osamu Iyama, Yuji Yoshino, \emph{Mutation in triangulated categories and rigid Cohen-Macaulay modules}, Invent. Math. 172 (2008), no. 1, 117--168. 



 \bibitem[Ji1]{J1}
Haibo Jin, \emph{Cohen-Macaulay differential graded modules and negative Calabi-Yau configurations}, arXiv:1812.03737.

\bibitem[Ji2]{J2}
Haibo Jin, \emph{Reductions of triangulated categories and simple-minded collections}, arXiv:1907.05114.

\bibitem[Jo]{Jo}
Peter J\o rgensen, \emph{Auslander-Reiten theory over topological spaces}, Comment. Math. Helv. 79 (2004), no. 1, 160--182.

\bibitem[K]{Keller05}
Bernhard Keller, \emph{On triangulated orbit categories}.
Doc. Math. 10 (2005), 551--581. 

\bibitem[KN]{KN}
Bernhard Keller, Pedro Nicol\'as, {\itshape Cluster hearts and cluster tilting objects}, work in preparation.


\bibitem[KV]{KV}
Bernhard Keller, Dieter Vossieck, \emph{Aisles in derived categories}, Bull. Soc. Math. Belg. S\'er. A 40 (1988), no. 2, 239--253.


\bibitem[KYZ]{KYZ}
Bernhard Keller, Dong Yang, Guodong Zhou, \emph{The Hall algebra of a spherical object}, J. Lond. Math. Soc. (2) 80 (2009), no. 3, 771--784.

\bibitem[KYa]{KY}
Steffen Koenig, Dong Yang, \emph{Silting objects, simple-minded collections, $t$-structures and co-$t$-structures for finite-dimensional algebras},
Doc. Math. 19 (2014), 403--438. 

\bibitem[KYu]{Yuming}
Steffen Koenig, Yuming Liu, \emph{Simple-minded systems in stable module categories}
Q. J. Math. 63 (2012), no. 3, 653--674.

\bibitem[Ric]{Ric}
Jeremy Rickard, \emph{Equivalences of derived categories for symmetric algebras}, J. Algebra 257 (2002), no. 2, 460--481. 


\bibitem[Rie]{Riedtmann}
Christine Riedtmann, \emph{Representation-finite self-injective algebras of class $A_{n}$}, Representation theory, II (Proc. Second Internat. Conf., Carleton Univ., Ottawa, Ont., 1979), pp. 449--520, Lecture Notes in Math., 832, Springer, Berlin, 1980. 

\bibitem[STW]{STW}
Christian Stump, Hugh Thomas, Nathan Williams,
\emph{Cataland: Why the Fuss?}, arXiv:1503.00710.

\bibitem[T]{T}
Hugh Thomas, \emph{Defining an m-cluster category}, J. Algebra 318 (2007), no. 1, 37--46.

\bibitem[Z]{Z}
Bin, Zhu, \emph{Generalized cluster complexes via quiver representations}, J. Algebraic Combin. 27 (2008), no. 1, 35--54.

\end{thebibliography}
\end{document}